\documentclass[11pt]{article}

\usepackage[T1]{fontenc}
\usepackage[utf8]{inputenc}
\usepackage{lmodern}
\usepackage[margin=1.08in]{geometry}
\usepackage{amsmath,amssymb,amsthm,mathtools}
\usepackage{microtype}
\usepackage[colorlinks=true,linkcolor=blue,citecolor=blue,urlcolor=blue,pdftitle={Tur\'an-good monotonicity thresholds},pdfauthor={Yuanpei Wang, Liying Kang, and Xiaomiao Zhao}]{hyperref}

\allowdisplaybreaks
\renewcommand{\qedsymbol}{$\blacksquare$}
\numberwithin{equation}{section}

\newtheorem{theorem}{Theorem}[section]
\newtheorem{lemma}[theorem]{Lemma}

\theoremstyle{definition}

\newtheorem{problem}[theorem]{Problem}
\theoremstyle{remark}

\newcommand{\inj}{\operatorname{inj}}
\newcommand{\Inj}{\operatorname{Inj}}
\newcommand{\ex}{\operatorname{ex}}
\newcommand{\Aut}{\operatorname{Aut}}
\newcommand{\Col}{\operatorname{Col}}
\newcommand{\dist}{\operatorname{dist}}
\newcommand{\E}{\mathbb E}
\newcommand{\Prb}{\mathbb P}
\newcommand{\cP}{\mathcal P}
\newcommand{\cE}{\mathcal E}

\title{%
	\fontsize{20}{24}\selectfont
	\bfseries
	Tur\'an-good monotonicity thresholds%
}
\author{%
Yuanpei Wang\textsuperscript{*}
\qquad
Liying Kang\textsuperscript{\(\dagger\)}
\qquad
Xiamiao Zhao\textsuperscript{\(\ddagger\)}%
}
\date{}

\begin{document}

\maketitle
\begingroup
\renewcommand{\thefootnote}{\fnsymbol{footnote}}
\footnotetext[1]{Department of Mathematics, Shanghai University, Shanghai 200444, P.R. China. Supported by the China Scholarship Council (No. 202506890067). Email: boyuan@shu.edu.cn.}
\footnotetext[2]{Department of Mathematics, Shanghai University, Shanghai 200444, P.R. China, and Newtouch Center for Mathematics of Shanghai University, Shanghai 200444, P.R. China. Supported by the National Natural Science Foundation of China (Grant Nos. 12571375 and 12331012). Email: lykang@shu.edu.cn.}
\footnotetext[3]{Corresponding author. Department of Mathematical Sciences, Tsinghua University, Beijing 100084, P.R. China. Supported by the China Scholarship Council (No. 202506210250) and the National Natural Science Foundation of China (Grant No. 12571372). Email: zxm23@mails.tsinghua.edu.cn.}
\endgroup

\begin{abstract}
A graph $H$ is $K_{r+1}$-Tur\'an-good if, for every sufficiently large $n$, the Tur\'an graph $T_r(n)$ maximizes the number of copies of $H$ among all $n$-vertex $K_{r+1}$-free graphs. It is strictly $K_{r+1}$-Tur\'an-good if $T_r(n)$ is the unique extremal graph.

Morrison, Nir, Norin, Rz\k{a}\.zewski and Wesolek [\emph{JCTB}, 2023] proved that every graph $H$ is $K_{r+1}$-Tur\'an-good whenever $r\ge 300v(H)^9$. They raised the following two questions:
\begin{enumerate}
\item Can the sufficient condition $r\ge 300v(H)^9$ be reduced to a condition of quadratic order in $v(H)$?
\item Is the Turán-good property monotone in $r$? More precisely, if a graph $H$ is $K_r$-Turán-good, must it also be $K_{r+1}$-Turán-good?
\end{enumerate}

We affirmatively resolve the first question and derive an even stronger bound linear in the edge number: every graph $H$ with at least one edge is strictly $K_{r+1}$-Tur\'an-good and $K_{r+1}$-Tur\'an-stable whenever $r\ge 168e(H)$. This condition is quadratic in $v(H)$ for arbitrary graphs and linear in $v(H)$ for every sparse graph family with $e(H)=O(v(H))$.

We answer the second question negatively. For every $r\ge3$, there exists a graph that is strictly $K_r$-Tur\'an-good but not $K_{r+1}$-Tur\'an-good. More quantitatively, for every sufficiently large $h$, there exists a graph $H$ with $v(H)\le h$ and an integer $r=h-2\sqrt h+O(1)$ such that $H$ is strictly $K_r$-Tur\'an-good but not $K_{r+1}$-Tur\'an-good.

The monotonicity threshold $\lambda(H)$ is the least integer $R\ge 2$ such that, for every $r\ge R$, the graph $H$ is $K_{r+1}$-Tur\'an-good whenever it is $K_r$-Tur\'an-good. For
\[
\lambda_{\max}(h)=\max\{\lambda(H)\mid v(H)\le h\},
\]
our two results yield
\[
h-2\sqrt h-O(1)\le \lambda_{\max}(h)\le 84h^2.
\]
\end{abstract}


\section{Introduction}

Throughout the paper, all graphs are finite and simple. For a graph $G$, let $V(G)$ and $E(G)$ denote its vertex and edge sets, and let $v(G):=|V(G)|$ and $e(G):=|E(G)|$. For $u\in V(G)$, let $N_G(u)$ be the \emph{neighborhood} of $u$ and let $d_G(u):=|N_G(u)|$ be its \emph{degree}. Let $\delta(G)$, $\Delta(G)$, and $\chi(G)$ denote the \emph{minimum degree}, \emph{maximum degree}, and \emph{chromatic number} of $G$, respectively.

For fixed graphs $H$ and $F$, the \emph{generalized Tur\'an number} is
\[
\ex(n,H,F):=\max\{N(H,G):|V(G)|=n,\ G\text{ is }F\text{-free}\},
\]
where $N(H,G)$ denotes the number of unlabeled copies of $H$ in $G$. The generalized Tur\'an problem was introduced by Alon and Shikhelman \cite{AlonShikhelman}. For further results, we refer to the survey of Gerbner and Palmer \cite{GerbnerPalmerSurvey}. The case $H=K_2$ is the classical Tur\'an problem.

Let $T_r(n)$ denote the \emph{Tur\'an graph}, the complete $r$-partite graph on $n$ vertices whose part sizes differ by at most one. Following Gerbner and Palmer \cite{GerbnerPalmerCounting,GerbnerPalmerExact}, a graph $H$ is \emph{$K_{r+1}$-Tur\'an-good} if $\ex(n,H,K_{r+1})=N(H,T_r(n))$ for every sufficiently large $n$. A graph attaining this maximum is called an \emph{extremal graph}. If $T_r(n)$ is the unique extremal graph for every sufficiently large $n$, then $H$ is \emph{strictly} $K_{r+1}$-Tur\'an-good. Zykov's theorem \cite{Zykov} shows that every clique $K_t$ with $t\le r$ is strictly $K_{r+1}$-Tur\'an-good. Further exact and stability results for several graph classes were obtained by Ma and Qiu \cite{MaQiu}.

An \emph{injective homomorphism} from $H$ to $G$ is an injective map $\varphi:V(H)\to V(G)$ such that $uv\in E(H)$ implies $\varphi(u)\varphi(v)\in E(G)$. Let $\Inj(H,G)$ be the set of all such maps, and let $\inj(H,G):=|\Inj(H,G)|$. Since $\inj(H,G)=|\Aut(H)|N(H,G)$, where $\Aut(H)$ is the \emph{automorphism group} of $H$, maximizing $\inj(H,G)$ is equivalent to maximizing $N(H,G)$. For a nonnegative integer $a$ and an integer $b\ge 0$, define $(a)_b:=a(a-1)\cdots(a-b+1)$.

The \emph{edit distance} of two $n$-vertex graphs $G$ and $G'$ is the minimum number of edges that need to be added to or deleted from $G$ to obtain a graph isomorphic to $G'$. A graph $H$ is \emph{$K_{r+1}$-Tur\'an-stable} if, for every $\varepsilon>0$, there are $\delta>0$ and $n_0$ such that every $n\ge n_0$ and every $n$-vertex $K_{r+1}$-free graph $G$ satisfying $N(H,G)\ge \ex(n,H,K_{r+1})-\delta n^{v(H)}$ has edit distance at most $\varepsilon n^2$ from $T_r(n)$.

Gerbner and Palmer \cite{GerbnerPalmerCounting,GerbnerPalmerExact} conjectured that every graph is eventually Tur\'an-good. Morrison, Nir, Norin, Rz\k{a}\.zewski and Wesolek~\cite{MorrisonEtAl} proved the following theorem.

\begin{theorem}[Morrison, Nir, Norin, Rz\k{a}\.zewski and Wesolek~\cite{MorrisonEtAl}]
\label{thm:MNNRW}
Let $H$ be a graph and let $r\ge 300v(H)^9$. Then $H$ is $K_{r+1}$-Tur\'an-good.
\end{theorem}

In the same work, they raised the following two problems concerning the dependence on $r$.

\begin{problem}
\label{prob:quadratic-threshold}
Can the sufficient condition $r\ge 300v(H)^9$ in Theorem \ref{thm:MNNRW} be replaced by a condition of quadratic order in $v(H)$?
\end{problem}

\begin{problem}
\label{prob:monotonicity}
Fix a graph $H$ and suppose that $H$ is $K_r$-Tur\'an-good. Does it follow that $H$ is also $K_{r+1}$-Tur\'an-good?
\end{problem}

Our first result answers Problem \ref{prob:quadratic-threshold} affirmatively and gives a stronger bound in terms of the number of edges. Under the same hypothesis $r\ge 300v(H)^9$, Gerbner and Hama Karim \cite{GerbnerHamaKarim} proved that every graph $H$ with at least one edge is $K_{r+1}$-Tur\'an-stable. Chen and Liu \cite{ChenLiu} subsequently strengthened this stability result and extended it to hypergraphs.

Every edgeless graph is trivially $K_{r+1}$-Tur\'an-good, but not strictly so. Isolated vertices may be removed without changing the relevant extremal conclusions. Indeed, if $H=H_0\cup tK_1$, where $H_0$ has no isolated vertices, then
\[
\inj(H,G)=(n-v(H_0))_t\,\inj(H_0,G)
\]
for every $n$-vertex graph $G$. The factor is independent of $G$, and one of $H,H_0$ is $K_{r+1}$-Tur\'an-stable if and only if the other is.

Morrison, Nir, Norin, Rz\k{a}\.zewski and Wesolek \cite{MorrisonEtAl} observed that improving the bound in Theorem \ref{thm:MNNRW} to quadratic order in $v(H)$ would likely require additional ideas. The following theorem replaces $300v(H)^9$ by $168e(H)$ and simultaneously gives strictness and stability.

\begin{theorem}
\label{thm:main}
For every graph $H$ with at least one edge and every integer $r\ge 168e(H)$, the graph $H$ is strictly $K_{r+1}$-Tur\'an-good and $K_{r+1}$-Tur\'an-stable.
\end{theorem}

Since $e(H)\le\binom{v(H)}2$, Theorem \ref{thm:main} gives a quadratic sufficient condition in $v(H)$. More strongly, for every graph family satisfying $e(H)=O(v(H))$, it gives a linear sufficient condition in the number of vertices.

Our second result answers Problem \ref{prob:monotonicity} negatively. For our constructions, the lower-level extremal property follows from the following theorem.

\begin{theorem}[Gy\H{o}ri, Pach and Simonovits~\cite{GyoriPachSimonovits}]
\label{thm:GPS}
Let $r\ge 3$, and let $H$ be an $(r-1)$-partite graph with $h>r-1$ vertices. Suppose that $H$ contains $\lfloor h/(r-1)\rfloor$ pairwise vertex-disjoint copies of $K_{r-1}$ and that any two vertices $x,y$ in the same connected component can be joined by a sequence $Q_1,\ldots,Q_s$ of copies of $K_{r-1}$ such that $x\in Q_1$, $y\in Q_s$, and $|Q_i\cap Q_{i+1}|=r-2$ for $1\le i<s$. Then $H$ is strictly $K_r$-Tur\'an-good.
\end{theorem}

For $r\ge 3$ and $t\ge 2$, Theorem \ref{thm:GPS} implies that $T_{r-1}((r-1)t)$ is strictly $K_r$-Tur\'an-good.

\begin{theorem}
\label{thm:multipartite-counterexample}
For every $r\ge 3$, there exists $t_0=t_0(r)\ge 2$ such that, for every $t\ge t_0$, the graph $T_{r-1}((r-1)t)$ is not $K_{r+1}$-Tur\'an-good.
\end{theorem}

Together with Theorem \ref{thm:GPS}, Theorem \ref{thm:multipartite-counterexample} shows that, for every $r\ge3$, there exists a graph that is strictly $K_r$-Tur\'an-good but not $K_{r+1}$-Tur\'an-good. Thus monotonicity can fail at every level $r\ge3$.

This leads naturally to the question of when monotonicity is eventually guaranteed. Define the \emph{monotonicity threshold} $\lambda(H)$ to be the smallest integer $R\ge2$ such that, for every $r\ge R$, if $H$ is $K_r$-Tur\'an-good, then it is $K_{r+1}$-Tur\'an-good, and define
\[
\lambda_{\max}(h):=\max\{\lambda(H):H\text{ is a graph and }v(H)\le h\}.
\]
Theorem \ref{thm:MNNRW} gives $\lambda(H)\le300v(H)^9$ and hence $\lambda_{\max}(h)\le300h^9$, while Theorem \ref{thm:main} improves this to $\lambda_{\max}(h)\le84h^2$.

Theorem \ref{thm:multipartite-counterexample} establishes failure at every level $r$, but it is not optimized for the order of the counterexample. To obtain a strong lower bound on $\lambda_{\max}(h)$, we next construct counterexamples in which $r$ is as large as possible relative to the number of vertices.

For $r\ge 3$ and $1\le t<r-1$, let $F_{r,t}$ have vertex set $C\cup Z$, where $C=\{y_1,\ldots,y_{r-1}\}$ induces $K_{r-1}$, the set $Z=\{z_1,\ldots,z_t\}$ is independent, and $N_{F_{r,t}}(z_i)=C\setminus\{y_i\}$.
The graph $F_{r,t}$ is also strictly $K_r$-Tur\'an-good by Theorem \ref{thm:GPS}.

\begin{theorem}
\label{thm:core-satellite}
For every $r\ge 17$, there exists $t_0=t_0(r)<r-1$ such that, for every $t$ with $t_0\le t<r-1$, the graph $F_{r,t}$ is not $K_{r+1}$-Tur\'an-good.
\end{theorem}

Optimizing the counterexamples in Theorem \ref{thm:core-satellite} shows that monotonicity can fail as far as $h-2\sqrt h-O(1)$. Together with the upper bound above, this yields the following estimate.

\begin{theorem}
\label{thm:threshold-range}
For every sufficiently large integer $h$,
\[
h-2\sqrt h-O(1)\le \lambda_{\max}(h)\le 84h^2.
\]
\end{theorem}

The rest of the paper is organized as follows. Section \ref{sec:main-proof} proves Theorem \ref{thm:main}. Sections \ref{sec:multipartite-proof} and \ref{sec:core-satellite-proof} prove Theorems \ref{thm:multipartite-counterexample} and \ref{thm:core-satellite}, respectively. We finish with concluding remarks.

\section{Proof of Theorem \ref{thm:main}}
\label{sec:main-proof}

The proof uses the following theorem of F\"uredi to obtain an $r$-partite subgraph of a $K_{r+1}$-free graph.

\begin{theorem}[F\"uredi \cite{Furedi}]
\label{thm:Furedi}
Let $G$ be an $n$-vertex $K_{r+1}$-free graph. Then $G$ contains an $r$-partite subgraph $G_0$ such that
\[
e(G)-e(G_0)\le e(T_r(n))-e(G).
\]
The graph $G_0$ may be taken to be spanning by adding omitted vertices as isolated vertices.
\end{theorem}

After deleting isolated vertices from $H$, write
\[
h:=v(H),\qquad m:=e(H),\qquad \Delta:=\Delta(H),
\]
and define
\[
J:=\sum_{u\in V(H)}d_H(u)(d_H(u)-1).
\]
We repeatedly use
\[
h\le 2m,\qquad J\le 2m(\Delta-1),\qquad J\le m(m-1),\qquad \Delta\le m.
\]
The first bound for $J$ follows from $d_H(u)-1\le \Delta-1$, while the second follows because $J$ counts ordered pairs of distinct intersecting edges. All asymptotic notation refers to $n\to\infty$ with $H$ and $r$ fixed. Constants in $O_{H,r}(\cdot)$ may depend on $H$ and $r$.

For a graph $F$, let $\Col_r(F)$ be the set of proper colorings of $F$ with labeled colors $1,\ldots,r$, and define
\[
p_r(F):=\frac{|\Col_r(F)|}{r^{v(F)}}.
\]
Thus $p_r(F)$ is the probability that a uniformly random $r$-coloring of $F$ is proper. Label the parts of $T_r(n)$ by $A_1,\ldots,A_r$, and write $a_i:=|A_i|=n/r+O_r(1)$. Every injective homomorphism $F\to T_r(n)$ induces a unique coloring $\varphi\in\Col_r(F)$ according to the parts containing the vertex images. Conversely, if $s_i(\varphi):=|\varphi^{-1}(i)|$, then exactly $\prod_{i=1}^r(a_i)_{s_i(\varphi)}$ injective homomorphisms induce $\varphi$. Since $\sum_{i=1}^r s_i(\varphi)=v(F)$, each such product is $r^{-v(F)}n^{v(F)}+O_{F,r}(n^{v(F)-1})$. Therefore
\begin{equation}
\label{eq:Turan-leading-term}
\inj(F,T_r(n))
=\sum_{\varphi\in\Col_r(F)}\prod_{i=1}^r(a_i)_{s_i(\varphi)}
=p_r(F)n^{v(F)}+O_{F,r}(n^{v(F)-1}).
\end{equation}

\subsection{The strict Tur\'an-good conclusion}

We prove the non-strict extremal inequality by induction on $m=e(H)$, and prove uniqueness whenever $m\ge1$. If $m=0$, then $H$ is edgeless and $\inj(H,G)=(n)_{v(H)}$ for every $n$-vertex graph $G$. Now assume that $m\ge1$ and that the inequality holds for every graph with fewer than $m$ edges. Whenever a graph $F$ arising in the induction has isolated vertices, write $F=F_0\cup tK_1$, where $F_0$ has no isolated vertices. Since $\inj(F,G)=(n-v(F_0))_t\inj(F_0,G)$ for every $n$-vertex graph $G$, it suffices to apply the induction hypothesis to $F_0$.

Assume that $r\ge 168m$, and let $G$ be an $n$-vertex $K_{r+1}$-free graph maximizing $\inj(H,G)$. We first bound the number of nonneighbors of each vertex of $G$. We then use this local estimate to bound the proportion of ordered pairs of distinct vertices that are nonedges of $G$.

\begin{lemma}
\label{lem:calibration}
We have
\[
\delta(G)\ge
\left(1-\frac{31m}{10r\Delta}-o(1)\right)n.
\]
\end{lemma}

\begin{proof}
For $v\in V(G)$, let $I_G(v)$ be the number of injective homomorphisms $H\to G$ whose image contains $v$. For $u\in V(H)$, let $I_{G,u}(v)$ be the number of injective homomorphisms $\varphi:H\to G$ satisfying $\varphi(u)=v$. Since an injective map has a unique preimage of $v$,
\begin{equation}
\label{eq:rooted-decompose}
I_G(v)=\sum_{u\in V(H)}I_{G,u}(v).
\end{equation}

Choose $v_0\in V(G)$ maximizing $I_G(v_0)$. Double counting pairs consisting of an injective homomorphism and a vertex in its image gives
\[
I_G(v_0)\ge \frac{h}{n}\inj(H,G).
\]
By the extremality of $G$ and \eqref{eq:Turan-leading-term},
\[
I_G(v_0)\ge hp_r(H)n^{h-1}-O_{H,r}(n^{h-2}).
\]

Fix $v\in V(G)$. If $v\ne v_0$, let $G'$ be obtained by replacing $v$ with a twin of $v_0$: delete all edges incident with $v$, and then join $v$ to every vertex of $N_G(v_0)\setminus\{v\}$. In particular, $vv_0\notin E(G')$. The graph $G'$ remains $K_{r+1}$-free, because any new copy of $K_{r+1}$ would contain $v$ but not $v_0$, and replacing $v$ by $v_0$ would produce a copy of $K_{r+1}$ in $G$.

Every injective homomorphism $H\to G$ whose image contains $v_0$ but not $v$ yields an injective homomorphism $H\to G'$ after replacing $v_0$ by $v$. Only injective homomorphisms whose images contain $v$ can be lost when passing from $G$ to $G'$. Moreover, at most $h(h-1)(n-2)_{h-2}=O_H(n^{h-2})$ injective homomorphisms have images containing both $v$ and $v_0$. Hence
\[
\inj(H,G')-\inj(H,G)\ge I_G(v_0)-I_G(v)-O_H(n^{h-2}).
\]
The extremality of $G$ gives
\begin{equation}
\label{eq:all-rooted-lower}
I_G(v)\ge I_G(v_0)-O_H(n^{h-2})
\ge hp_r(H)n^{h-1}-O_{H,r}(n^{h-2}).
\end{equation}

Write $d_G(v)=(1-\eta)n$. Fix $u\in V(H)$. If an injective homomorphism maps $u$ to $v$, then the $d_H(u)$ neighbors of $u$ must be mapped to distinct vertices of $N_G(v)$. Ignoring all remaining edge constraints gives
\begin{equation}
\label{eq:rooted-degree-upper}
I_{G,u}(v)\le (d_G(v))_{d_H(u)}n^{h-1-d_H(u)}
\le (1-\eta)^{d_H(u)}n^{h-1}.
\end{equation}
On the other hand, restricting such an embedding to $H-u$ gives an injective homomorphism from $H-u$ to $G-v$. Since $e(H-u)=m-d_H(u)<m$, the induction hypothesis gives
\[
I_{G,u}(v)\le \inj(H-u,G-v)
\le \inj(H-u,T_r(n-1)).
\]
Using \eqref{eq:Turan-leading-term},
\begin{equation}
\label{eq:rooted-inductive-upper}
I_{G,u}(v)\le p_r(H-u)n^{h-1}+O_{H,r}(n^{h-2}).
\end{equation}
Combining \eqref{eq:rooted-degree-upper} and \eqref{eq:rooted-inductive-upper}, then summing over $u$ and using \eqref{eq:rooted-decompose} and \eqref{eq:all-rooted-lower}, yields
\begin{equation}
\label{eq:calibration-master}
hp_r(H)\le \sum_{u\in V(H)}\min\{p_r(H-u),(1-\eta)^{d_H(u)}\}+o(1).
\end{equation}

Restricting a uniformly random $r$-coloring of $H$ to $H-u$ gives a uniformly random $r$-coloring of $H-u$. Since every proper coloring of $H$ is proper on $H-u$, the difference $p_r(H-u)-p_r(H)$ is the probability that $H-u$ is properly colored and some edge incident with $u$ is monochromatic. Each such edge is monochromatic with probability $1/r$, so the union bound gives
\[
0\le p_r(H-u)-p_r(H)\le \frac{d_H(u)}{r}.
\]
Consequently,
\begin{equation}
\label{eq:sum-deletion-probabilities}
\sum_{u\in V(H)}p_r(H-u)
\le hp_r(H)+\frac1r\sum_{u\in V(H)}d_H(u)
=hp_r(H)+\frac{2m}{r}.
\end{equation}
Since $H-u$ has $m-d_H(u)$ edges, the union bound also gives
\begin{equation}
\label{eq:deletion-probability-lower}
p_r(H-u)\ge 1-\frac{m-d_H(u)}{r}.
\end{equation}

Let $u_*$ be a vertex of degree $\Delta$. Suppose for contradiction that $\Delta\eta\ge 31m/(10r)$. Since $1-\eta\le (1+\eta)^{-1}$ and $(1+\eta)^\Delta\ge 1+\Delta\eta$, this assumption and $r\ge 168m$ give
\begin{equation}
\label{eq:degree-loss-calibrated}
1-(1-\eta)^\Delta
\ge \frac{31m/(10r)}{1+31m/(10r)}
\ge \frac{5208m}{1711r}
>\left(3+\frac1{23}\right)\frac mr.
\end{equation}
Using \eqref{eq:deletion-probability-lower} for $u_*$, we obtain
\[
p_r(H-u_*)-(1-\eta)^\Delta\ge 1-\frac{m-\Delta}{r}-(1-\eta)^\Delta\ge \left(2+\frac1{23}\right)\frac mr.
\]
Therefore
\[
\min\{p_r(H-u_*),(1-\eta)^\Delta\}
\le p_r(H-u_*)-\left(2+\frac1{23}\right)\frac mr.
\]
For all other vertices, we use the trivial bound by $p_r(H-u)$. Together with \eqref{eq:sum-deletion-probabilities}, this gives
\[
\sum_{u\in V(H)}\min\{p_r(H-u),(1-\eta)^{d_H(u)}\}
\le hp_r(H)-\frac{m}{23r}.
\]
For sufficiently large $n$, this contradicts \eqref{eq:calibration-master}. Hence $\eta\le 31m/(10r\Delta)+o(1)$. Since $v\in V(G)$ was arbitrary, the claimed bound on $\delta(G)$ follows.
\end{proof}

Lemma \ref{lem:calibration} gives a local bound on the number of nonneighbors of each vertex. We next estimate the global proportion of nonedges.

For an $n$-vertex graph $Q$, define its \emph{ordered nonedge density} by $q(Q):=\tfrac{n(n-1)-2e(Q)}{n(n-1)}$.

\begin{lemma}
\label{lem:nonedge-density}
We have
\[
q(G)\le \frac1r\left(1+\frac{29m}{4r}\right)+o(1).
\]
\end{lemma}

\begin{proof}
Choose a uniformly random injection $\varphi:V(H)\to V(G)$. For each $e=xy\in E(H)$, let $X_e$ be the indicator of the event $\varphi(x)\varphi(y)\notin E(G)$, and define
\[
Y:=\sum_{e\in E(H)}X_e.
\]
Let $q:=q(G)$. Then $\E X_e=q$ and hence $\E Y=mq$.

Thus $Y$ counts the edges of $H$ mapped to nonedges of $G$. Our aim is to bound $q$ by comparing upper and lower bounds on $\Prb(Y>0)$. The lower bound will follow from the second-moment inequality $\Prb(Y>0)\ge (\E Y)^2/\E Y^2$, so we estimate $\E Y^2$. Since each $X_e$ is an indicator variable,
\[
Y^2=\sum_{e\in E(H)}X_e+
\sum_{\substack{e,f\in E(H)\\ e\ne f}}X_eX_f.
\]
Thus we must bound the probability that two distinct edges $e$ and $f$ of $H$ are both mapped to nonedges of $G$. They are either vertex-disjoint or share exactly one endpoint.

Suppose first that $e$ and $f$ are vertex-disjoint. Conditional on the ordered image of $e$, the edge $f$ is mapped uniformly to one of the $(n-2)(n-3)$ ordered pairs avoiding those two vertices, among which at most $qn(n-1)$ are nonedges. Therefore
\[
\Prb(X_e=X_f=1)
\le q\,\frac{qn(n-1)}{(n-2)(n-3)}
=q^2+O_H(1/n).
\]
Now suppose that $e$ and $f$ share a vertex $w$, and condition on $\varphi(w)=z$. Define
\[
\overline d(z):=n-1-d_G(z).
\]
By Lemma \ref{lem:calibration},
$\overline d(z)\le \bigl(31m/(10r\Delta)+o(1)\bigr)n$. Therefore
\[
\Prb(X_e=X_f=1\mid \varphi(w)=z)
=\frac{(\overline d(z))_2}{(n-1)_2}
\le \left(\frac{31m}{10r\Delta}+o(1)\right)\frac{\overline d(z)}{n-1}.
\]
Since $z$ is uniformly distributed over $V(G)$,
\[
\E_z\frac{\overline d(z)}{n-1}
=\frac1n\sum_{z\in V(G)}\frac{n-1-d_G(z)}{n-1}
=\frac{n(n-1)-2e(G)}{n(n-1)}
=q.
\]
Taking the expectation over $z=\varphi(w)$ now gives
\[
\Prb(X_e=X_f=1)\le \left(\frac{31m}{10r\Delta}+o(1)\right)q.
\]
There are at most $m^2$ ordered disjoint edge pairs and exactly $J$ ordered intersecting edge pairs. Hence
\[
\E Y^2\le mq+m^2q^2+\frac{31Jm}{10r\Delta}q+o(1)
=mq\left(1+mq+\frac{31J}{10r\Delta}+o(1)\right).
\]
Since $J/m\le2\Delta$, we have $31J/(10r\Delta)\le31m/(5r)$. The second-moment inequality therefore gives
\begin{equation}
\label{eq:Y-positive-lower}
\Prb(Y>0)\ge \frac{(\E Y)^2}{\E Y^2}
\ge \frac{mq}{1+mq+31m/(5r)+o(1)}.
\end{equation}

The event $Y=0$ is precisely the event that $\varphi$ is an injective homomorphism. A uniformly random injection into $T_r(n)$ fails to be an injective homomorphism only if at least one edge of $H$ has both endpoints in one part. The union bound gives
\[
\inj(H,T_r(n))\ge \left(1-\frac{m}{r}-o(1)\right)(n)_h.
\]
Since $G$ is extremal,
\begin{equation}
\label{eq:Y-positive-upper}
\Prb(Y>0)=1-\frac{\inj(H,G)}{(n)_h}\le \frac{m}{r}+o(1).
\end{equation}
Combining \eqref{eq:Y-positive-lower} and \eqref{eq:Y-positive-upper} gives
\[
q(r-m)\le 1+\frac{31m}{5r}+o(1).
\]
Since $r\ge 168m$, we have $m/r\le 1/168$. Hence
\[
q\le \frac1r\frac{1+31m/(5r)}{1-m/r}+o(1)
=\frac1r\left(1+\frac{36m/(5r)}{1-m/r}\right)+o(1)
\le \frac1r\left(1+\frac{29m}{4r}\right)+o(1).
\]
This proves the lemma.
\end{proof}

Among all spanning $r$-partite subgraphs of $G$, choose $G_0$ with the maximum number of edges. Let $\cP=(A_1,\ldots,A_r)$ be its vertex partition, and let $K_{\cP}:=K_{A_1,\ldots,A_r}$ be the complete $r$-partite graph with parts $A_1,\ldots,A_r$. Thus $G_0$ consists of all edges of $G$ whose endpoints lie in different parts of $\cP$. Define
\[
L:=|E(G)\setminus E(G_0)|,\qquad
M:=|E(K_{\cP})\setminus E(G_0)|,
\qquad
B:=e(T_r(n))-e(K_{\cP}).
\]
By Theorem \ref{thm:Furedi}, $L\le e(T_r(n))-e(G)=B+M-L$, and hence
\begin{equation}
\label{eq:Furedi-assembly}
2L\le B+M.
\end{equation}

The choice of $G_0$ means that $\cP$ maximizes the number of crossing edges of $G$. For $v\in V(G)$ and $A\subseteq V(G)$, let $e_G(v,A):=|N_G(v)\cap A|$. If $v\in A_i$, moving $v$ from $A_i$ to $A_j$ cannot increase this number. Thus $e_G(v,A_i)\le e_G(v,A_j)$ for every $j\ne i$, and hence $e_G(v,A_i)\le d_G(v)/r\le n/r$. Since $d_{G_0}(v)=d_G(v)-e_G(v,A_i)$, Lemma \ref{lem:calibration} yields
\begin{equation}
\label{eq:G0-min-degree}
\delta(G_0)\ge (1-\gamma-o(1))n,
\qquad
\gamma:=\frac1r+\frac{31m}{10r\Delta}.
\end{equation}
If $A_i\ne\varnothing$, choose $v\in A_i$. Since $d_{G_0}(v)\le n-|A_i|$,
\begin{equation}
\label{eq:part-size-upper}
\frac{|A_i|}{n}\le \gamma+o(1).
\end{equation}
Moreover, Theorem \ref{thm:Furedi} and Lemma \ref{lem:nonedge-density} imply
\begin{equation}
q(G_0)=q(G)+\frac{2L}{n(n-1)}\le 2q(G)-q(T_r(n))\le \frac1r\left(1+\frac{29m}{2r}\right)+o(1).
\label{eq:G0-nonedge-density}
\end{equation}

We next estimate the increase in $\inj(H,\cdot)$ obtained by adding one of the
$M$ edges in \mbox{$E(K_{\cP})\setminus E(G_0)$}.

\begin{lemma}
\label{lem:completion}
Let $G_0\subseteq Q\subseteq K_{\cP}$, and let $xy\in E(K_{\cP})\setminus E(Q)$. Then
\[
\inj(H,Q+xy)-\inj(H,Q)
\ge 2m\left(1-mq(Q)-\frac{J}{m}\gamma-o(1)\right)n^{h-2}.
\]
\end{lemma}

\begin{proof}
Let
\[
E^{\mathrm{ord}}(H):=\bigcup_{uv\in E(H)}\{(u,v),(v,u)\}
\]
be the set of ordered edge pairs of $H$. Thus $|E^{\mathrm{ord}}(H)|=2m$.
For $(u,v)\in E^{\mathrm{ord}}(H)$, let $\cE_{u,v}$ be the set of injective homomorphisms $\varphi:H\to Q+xy$ satisfying $\varphi(u)=x$ and $\varphi(v)=y$. Every embedding counted by $\inj(H,Q+xy)-\inj(H,Q)$ uses the new edge $xy$, and injectivity gives a unique ordered preimage $(u,v)$. Thus
\begin{equation}
\label{eq:completion-orientations}
\inj(H,Q+xy)-\inj(H,Q)=\sum_{(u,v)\in E^{\mathrm{ord}}(H)}|\cE_{u,v}|.
\end{equation}

Fix $(u,v)\in E^{\mathrm{ord}}(H)$, set $\varphi(u)=x$ and $\varphi(v)=y$, and choose the other $h-2$ images uniformly among the $(n-2)_{h-2}$ injections into $V(Q)\setminus\{x,y\}$. Such an injection lies in $\cE_{u,v}$ precisely when no remaining edge maps to a nonedge of $Q$. We treat the failure probabilities separately for edges disjoint from $uv$ and edges incident with $uv$, then apply the union bound.

First let $ab$ be vertex-disjoint from $uv$. Then $(\varphi(a),\varphi(b))$ is uniform among the $(n-2)(n-3)$ ordered pairs of distinct vertices in $V(Q)\setminus\{x,y\}$. Since $Q$ has $q(Q)n(n-1)$ ordered nonedges,
\[
\Prb(\varphi(a)\varphi(b)\notin E(Q))
\le q(Q)\frac{n(n-1)}{(n-2)(n-3)}
=q(Q)+O(1/n).
\]
Next consider an edge incident with $uv$, say $uw$ with $w\ne v$. Since $\varphi(w)$ is uniform on $V(Q)\setminus\{x,y\}$ and $xy\notin E(Q)$, exactly $n-2-d_Q(x)$ of its $n-2$ possible images are nonneighbors of $x$. Hence
\[
\Prb(x\varphi(w)\notin E(Q))
=\frac{n-2-d_Q(x)}{n-2}\le \gamma+o(1),
\]
where we used $G_0\subseteq Q$ and \eqref{eq:G0-min-degree}. The same bound applies to an edge $vw$ incident with $v$.

The two classes have respective sizes
\[
a_{u,v}:=m-d_H(u)-d_H(v)+1,
\qquad
b_{u,v}:=d_H(u)+d_H(v)-2.
\]
For each $ab\in E(H)\setminus\{uv\}$, let $\mathcal B_{ab}$ be the event that $\varphi(a)\varphi(b)\notin E(Q)$, and write $\mathcal B:=\bigcup_{ab\in E(H)\setminus\{uv\}}\mathcal B_{ab}$.
The union bound and the two estimates above give
\[
\Prb(\mathcal B)
\le a_{u,v}\bigl(q(Q)+O(1/n)\bigr)
   +b_{u,v}\bigl(\gamma+o(1)\bigr)
=a_{u,v}q(Q)+b_{u,v}\gamma+o(1).
\]
Since $\cE_{u,v}$ consists exactly of the injections for which $\mathcal B$ does not occur, $|\cE_{u,v}|=(n-2)_{h-2}(1-\Prb(\mathcal B))$. Hence, using $a_{u,v}\le m$,
\begin{align*}
|\cE_{u,v}|
&\ge (n-2)_{h-2}\left(1-a_{u,v}q(Q)-b_{u,v}\gamma-o(1)\right)\\
&\ge (n-2)_{h-2}\left(1-mq(Q)-b_{u,v}\gamma-o(1)\right).
\end{align*}
We now sum over the $2m$ ordered pairs in $E^{\mathrm{ord}}(H)$. By the definition of $b_{u,v}$,
\begin{align*}
\sum_{(u,v)\in E^{\mathrm{ord}}(H)}b_{u,v}
&=\sum_{(u,v)\in E^{\mathrm{ord}}(H)}
  \bigl((d_H(u)-1)+(d_H(v)-1)\bigr)\\
&=2\sum_{w\in V(H)}d_H(w)(d_H(w)-1)=2J.
\end{align*}
Hence \eqref{eq:completion-orientations} gives
\[
\inj(H,Q+xy)-\inj(H,Q)
\ge 2m(n-2)_{h-2}
\left(1-mq(Q)-\frac{J}{m}\gamma-o(1)\right).
\]
Finally, $(n-2)_{h-2}=n^{h-2}(1+O_H(1/n))$, which proves the lemma.
\end{proof}

Order the $M$ edges in $E(K_{\cP})\setminus E(G_0)$ arbitrarily and add them successively. Every intermediate graph $Q$ satisfies $q(Q)\le q(G_0)$. By \eqref{eq:G0-nonedge-density},
\[
mq(Q)\le \frac mr\left(1+\frac{29m}{2r}\right)+o(1).
\]
Moreover, $J/m\le m$ and $J/m\le2\Delta$ give
\begin{equation}
\label{eq:Jgamma}
\frac{J}{m}\gamma
=\frac{J}{mr}+\frac{31J}{10r\Delta}
\le \frac mr+\frac{31m}{5r}=\frac{36m}{5r}.
\end{equation}
Since $m/r\le1/168$, we have
$\frac mr(1+\frac{29m}{2r})+\frac{36m}{5r}\le\frac{83m}{10r}$.
Consequently,
\begin{equation}
\label{eq:completion-total}
\inj(H,K_{\cP})-\inj(H,G_0)
\ge 2m\left(1-\frac{83m}{10r}-o(1)\right)Mn^{h-2}.
\end{equation}

We next transform $K_{\cP}$ into $T_r(n)$. Whenever two parts have sizes $a\ge b+2$, move one vertex from the larger part to the smaller part, changing $(a,b)$ to $(a-1,b+1)$. Repeating this operation produces $T_r(n)$. The following lemma estimates the change in the number of embeddings associated with one balancing step.

\begin{lemma}
\label{lem:two-part-finite-difference}
Let $F$ be a graph with $t:=e(F)\ge 1$. Let $a,b$ be nonnegative integers with $a\ge b+2$, and define $\ell:=a+b$ and $d:=a-b-1$. Then
\[
\left|\inj(F,K_{a-1,b+1})-\inj(F,K_{a,b})\right|
\le 4t^2d\ell^{v(F)-2}.
\]
\end{lemma}

\begin{proof}
If $b=0$, then $K_{a,0}$ is edgeless. Since $t\ge1$, every embedding of $F$ into $K_{a-1,1}$ sends an endpoint of some edge of $F$ to the singleton part. Choosing an oriented edge with this endpoint first gives
\[
\inj(F,K_{a-1,1})
\le 2t(a-1)_{v(F)-1}
\le 2t a^{v(F)-1}
\le 4t^2d\ell^{v(F)-2}.
\]
Here the last inequality follows from $\ell=a$, $d=a-1\ge a/2$, and $t\ge1$.

Assume now that $b\ge1$. If $\ell<v(F)$ or $F$ is not bipartite, the left-hand side is zero. We may therefore assume that $\ell\ge v(F)$ and that $F$ is bipartite.

Let $F^\circ$ be obtained from $F$ by deleting its isolated vertices. Since every vertex of $F^\circ$ is incident with an edge, $v(F^\circ)\le2t$. For $0\le x\le\ell$, define
\[
p(x):=\frac{\inj(F,K_{x,\ell-x})}{(\ell)_{v(F)}}
=\frac{\inj(F^\circ,K_{x,\ell-x})}{(\ell)_{v(F^\circ)}}.
\]
Indeed, every injective homomorphism of $F^\circ$ extends to $F$ in $(\ell-v(F^\circ))_{v(F)-v(F^\circ)}$ ways, and
$(\ell)_{v(F)}=(\ell)_{v(F^\circ)}(\ell-v(F^\circ))_{v(F)-v(F^\circ)}$.

Identify the vertex set of $K_{x,\ell-x}$ with $[\ell]:=\{1,\ldots,\ell\}$, and let $R_0$ be its part of size $x$. Then $p(x)$ is the probability that a uniformly random injection $\psi:V(F^\circ)\to[\ell]$ maps the endpoints of every edge of $F^\circ$ to different parts. Equivalently, fix an injection $\phi:V(F^\circ)\to[\ell]$ and choose a uniformly random $x$-subset $R\subseteq[\ell]$. To see this, let $\pi$ be a uniformly random permutation of $[\ell]$. Then $\psi=\pi\circ\phi$ is a uniformly random injection and $R=\pi^{-1}(R_0)$ is a uniformly random $x$-subset. Moreover, $\psi(u)\in R_0$ if and only if $\phi(u)\in R$. Hence
\[
p(x)=\Prb_R\bigl(
|\{\phi(u),\phi(v)\}\cap R|=1
\text{ for every }uv\in E(F^\circ)
\bigr).
\]

For $R\subseteq[\ell]$, define
\[
f(R):=
\begin{cases}
1, & \text{if }|\{\phi(u),\phi(v)\}\cap R|=1
     \text{ for every }uv\in E(F^\circ),\\
0, & \text{otherwise}.
\end{cases}
\]
Fix $1\le j\le\ell-1$. Choose a uniformly random $(j-1)$-subset $R$ of $[\ell]$, followed by an ordered pair of distinct elements $X,Y$ chosen uniformly from $[\ell]\setminus R$. The four sets $R\cup\{X,Y\}$, $R\cup\{X\}$, $R\cup\{Y\}$, and $R$ are uniform among the subsets of sizes $j+1,j,j,j-1$, respectively. Therefore
\[
p(j+1)-2p(j)+p(j-1)
=\E\bigl[f(R\cup\{X,Y\})-f(R\cup\{X\})-f(R\cup\{Y\})+f(R)\bigr].
\]
If $X\notin\phi(V(F^\circ))$, then $f(R\cup\{X,Y\})=f(R\cup\{Y\})$ and $f(R\cup\{X\})=f(R)$; the analogous equalities hold when $Y\notin\phi(V(F^\circ))$. Thus the expression inside the expectation vanishes unless $X,Y\in\phi(V(F^\circ))$, and its absolute value is at most $2$.

Since $(X,Y)$ is uniform among the $\ell(\ell-1)$ ordered pairs of distinct elements of $[\ell]$, we obtain
\begin{equation}
\label{eq:second-difference-p}
|p(j+1)-2p(j)+p(j-1)|
\le \frac{2v(F^\circ)(v(F^\circ)-1)}{\ell(\ell-1)}
\le 8\frac{t^2}{\ell^2}.
\end{equation}
Here the last inequality uses $v(F^\circ)\le2t$ and $\ell(\ell-1)\ge\ell^2/2$.

Define
\[
D_j:=p(j)-p(j-1),
\qquad
\alpha:=8\frac{t^2}{\ell^2}.
\]
Then \eqref{eq:second-difference-p} gives $|D_{j+1}-D_j|\le\alpha$. Interchanging the two parts gives $p(x)=p(\ell-x)$. Moreover, $\ell=a+b$ and $a\ge b+2$ imply $a\ge\ell/2+1$.

Suppose first that $\ell=2c+1$. Then $p(c+1)=p(c)$, so $D_{c+1}=0$. Since $d=2(a-c-1)$, telescoping from $D_{c+1}$ to $D_a$ gives
\[
|D_a|
\le\sum_{j=c+1}^{a-1}|D_{j+1}-D_j|
\le(a-c-1)\alpha
=\frac d2\alpha.
\]

Now suppose that $\ell=2c$. In this case $D_{c+1}=-D_c$, and hence $2|D_{c+1}|=|D_{c+1}-D_c|\le\alpha$. Since $d=2(a-c)-1$, we obtain
\[
|D_a|
\le |D_{c+1}|+\sum_{j=c+1}^{a-1}|D_{j+1}-D_j|
\le\frac{\alpha}{2}+(a-c-1)\alpha
=\frac d2\alpha.
\]
Thus, in either case,
\[
|p(a)-p(a-1)|\le \frac{d}{2}\alpha.
\]
Finally,
\[
\left|\inj(F,K_{a-1,b+1})-\inj(F,K_{a,b})\right|
=(\ell)_{v(F)}|p(a-1)-p(a)|.
\]
Since $(\ell)_{v(F)}\le\ell^{v(F)}$ and $(d/2)\alpha=4t^2d/\ell^2$, the desired bound follows.
\end{proof}

Lemma \ref{lem:two-part-finite-difference} gives the following estimate for a balancing step in a complete multipartite graph.

\begin{lemma}
\label{lem:balancing}
Let $K=K_{a_1,\ldots,a_r}$ have parts $A_1,\ldots,A_r$, where $a_k:=|A_k|$. Choose $i\ne j$ with $a_i\ge a_j+2$, and let $K'$ be obtained by moving one vertex from $A_i$ to $A_j$. Define
\[
d:=a_i-a_j-1=e(K')-e(K),\qquad
\ell:=a_i+a_j,\qquad
s:=\frac{\ell}{n},\qquad
\sigma_2:=\sum_{k=1}^r\left(\frac{a_k}{n}\right)^2.
\]
Then
\[
\inj(H,K')-\inj(H,K)
\ge 2mdn^{h-2}\left(1-5\frac{J}{m}s-5ms^2-m\sigma_2-O_H(1/n)\right).
\]
\end{lemma}

\begin{proof}
For $W\subseteq V(H)$, let $R(W)$ be the number of injective homomorphisms from $H-W$ to the complete $(r-2)$-partite graph induced by the parts other than $A_i,A_j$. Splitting each embedding according to $W=\varphi^{-1}(A_i\cup A_j)$ gives
\begin{align*}
\inj(H,K)
&=\sum_{W\subseteq V(H)}\inj(H[W],K_{a_i,a_j})R(W),\\
\inj(H,K')
&=\sum_{W\subseteq V(H)}\inj(H[W],K_{a_i-1,a_j+1})R(W).
\end{align*}
For each $W\subseteq V(H)$, define
\[
\Delta_W:=\inj(H[W],K_{a_i-1,a_j+1})-\inj(H[W],K_{a_i,a_j}).
\]
It follows that
\[
\inj(H,K')-\inj(H,K)=\sum_{W\subseteq V(H)}\Delta_WR(W).
\]
If $e(H[W])=0$, both embedding counts in the definition of $\Delta_W$ equal $(\ell)_{|W|}$, so $\Delta_W=0$. Define
\[
\Delta^{(1)}
:=\sum_{\substack{W\subseteq V(H)\\e(H[W])=1}}\Delta_WR(W),
\qquad
\Delta^{(\ge2)}
:=\sum_{\substack{W\subseteq V(H)\\e(H[W])\ge2}}\Delta_WR(W).
\]
Then
\[
\inj(H,K')-\inj(H,K)=\Delta^{(1)}+\Delta^{(\ge2)}.
\]

We first estimate $\Delta^{(1)}$. If $e(H[W])=1$, then $H[W]$ consists of one edge and isolated vertices. In this case
\[
\inj(H[W],K_{a_i,a_j})
=2a_ia_j(\ell-2)_{|W|-2}.
\]
Since $2\bigl((a_i-1)(a_j+1)-a_ia_j\bigr)=2d$, we have
\begin{equation}
\label{eq:one-edge-delta}
\Delta_W=2d(\ell-2)_{|W|-2}.
\end{equation}

Fix $uv\in E(H)$. To sum \eqref{eq:one-edge-delta} over the sets $W$ satisfying $E(H[W])=\{uv\}$, regard the images of $u$ and $v$ as fixed in the two selected parts and map the remaining $h-2$ vertices by a uniformly random injection into the remaining $n-2$ host vertices. Let $S$ be the set of the remaining $\ell-2$ vertices in $A_i\cup A_j$. Such an injection contributes to the desired sum provided none of the following events occurs.
\begin{enumerate}
\renewcommand{\labelenumi}{(\roman{enumi})}
\item Some vertex in $(N_H(u)\cup N_H(v))\setminus\{u,v\}$ is mapped into $S$. There are at most $d_H(u)+d_H(v)-2$ such vertices, and each is mapped into $S$ with probability $(\ell-2)/(n-2)=s+O(1/n)$. Hence this event has probability at most $\bigl(d_H(u)+d_H(v)-2\bigr)(s+O(1/n))$.
\item For some edge $ab\in E(H)$ vertex-disjoint from $uv$, both $a$ and $b$ are mapped into $S$. Then $a,b\in W$, so $E(H[W])\ne\{uv\}$. For each of at most $m$ such edges, the probability is $(\ell-2)_2/(n-2)_2=s^2+O(1/n)$, so the total probability is at most $m(s^2+O(1/n))$.
\item For some $k\notin\{i,j\}$, both endpoints of an edge of $H$ are mapped into $A_k$. For each edge, the probability is at most $\sum_{k\notin\{i,j\}}(a_k)_2/(n-2)_2\le \sigma_2+O(1/n)$. Summing over the $m$ edges gives $m(\sigma_2+O(1/n))$.
\end{enumerate}
By the union bound, the probability that at least one of these events occurs is at most
\[
\bigl(d_H(u)+d_H(v)-2\bigr)s+ms^2+m\sigma_2+O_H(1/n).
\]
Hence the terms corresponding to the sets $W$ satisfying $E(H[W])=\{uv\}$ sum to at least
\[
2d(n-2)_{h-2}
\left(1-\bigl(d_H(u)+d_H(v)-2\bigr)s-ms^2-m\sigma_2-O_H(1/n)\right).
\]
Summing over $uv\in E(H)$ and using
\[
\sum_{uv\in E(H)}\bigl(d_H(u)+d_H(v)-2\bigr)=J,
\qquad
(n-2)_{h-2}=n^{h-2}(1+O_H(1/n)),
\]
we obtain
\begin{equation}
\label{eq:balancing-one-edge}
\Delta^{(1)}
\ge 2mdn^{h-2}\left(1-\frac{J}{m}s-ms^2-m\sigma_2-O_H(1/n)\right).
\end{equation}

For $e(H[W])\ge2$, Lemma \ref{lem:two-part-finite-difference} gives $|\Delta_W|\le 4e(H[W])^2d\ell^{|W|-2}$. Therefore,
\[
\Delta^{(\ge2)}
\ge -4d\sum_{\substack{W\subseteq V(H)\\e(H[W])\ge2}}
e(H[W])^2\ell^{|W|-2}R(W).
\]
Since $R(W)\le(n-\ell)_{h-|W|}\le(n-\ell)^{h-|W|}$ and $s=\ell/n$,
it follows that
\[
\Delta^{(\ge2)}
\ge -4dn^{h-2}
\sum_{\substack{W\subseteq V(H)\\e(H[W])\ge2}}
e(H[W])^2s^{|W|-2}(1-s)^{h-|W|}.
\]
When $e(H[W])\ge2$, $e(H[W])^2\le2e(H[W])(e(H[W])-1)$, where $e(H[W])(e(H[W])-1)$ counts the ordered pairs of distinct edges in $H[W]$. For an ordered pair $(e,f)$ of distinct edges of $H$, let $U_{e,f}:=V(e)\cup V(f)$. Hence
\[
4\sum_{\substack{W\subseteq V(H)\\e(H[W])\ge2}}e(H[W])^2s^{|W|-2}(1-s)^{h-|W|}
\le 8\sum_{\substack{(e,f)\in E(H)^2\\e\ne f}}\ \sum_{U_{e,f}\subseteq W\subseteq V(H)}s^{|W|-2}(1-s)^{h-|W|}.
\]
For each fixed pair $(e,f)$, every set $W$ satisfying $U_{e,f}\subseteq W\subseteq V(H)$ is uniquely of the form $W=U_{e,f}\cup Z$, where $Z\subseteq V(H)\setminus U_{e,f}$. Hence
\[
\sum_{U_{e,f}\subseteq W\subseteq V(H)}s^{|W|-2}(1-s)^{h-|W|}
=s^{|U_{e,f}|-2}\sum_{Z\subseteq V(H)\setminus U_{e,f}}s^{|Z|}(1-s)^{h-|U_{e,f}|-|Z|}
=s^{|U_{e,f}|-2}.
\]
If $e$ and $f$ share one endpoint, then $|U_{e,f}|=3$; otherwise $|U_{e,f}|=4$. There are $J$ ordered pairs of the first type and at most $m^2$ of the second. Therefore,
\begin{equation}
\label{eq:balancing-multiple-edge}
\Delta^{(\ge2)}\ge -8dn^{h-2}(Js+m^2s^2).
\end{equation}
Combining \eqref{eq:balancing-one-edge} and \eqref{eq:balancing-multiple-edge} gives
\[
\inj(H,K')-\inj(H,K)
\ge2mdn^{h-2}
\left(1-5\frac{J}{m}s-5ms^2-m\sigma_2-O_H(1/n)\right),
\]
as required.
\end{proof}

Balancing $K_{\cP}$ into $T_r(n)$ increases the number of edges by $B=e(T_r(n))-e(K_{\cP})$. Throughout this process, neither the largest part size nor $\sigma_2$ increases. Hence, by \eqref{eq:part-size-upper}, every balancing step satisfies $s\le 2\gamma+o(1)$. Together with \eqref{eq:Jgamma}, this gives
\[
\frac{J}{m}s\le 2\frac{J}{m}\gamma+o(1)
\le \frac{72}{5}\frac{m}{r}+o(1).
\]
Every ordered pair of distinct vertices within the same initial part is a nonedge of $G_0$. Since $\sigma_2$ never increases, \eqref{eq:G0-nonedge-density} gives, at every balancing step,
\[
\sigma_2
\le \frac{\sum_{i=1}^r(|A_i|)_2}{n(n-1)}+O(1/n)
\le q(G_0)+O(1/n)
\le \frac1r\left(1+\frac{29m}{2r}\right)+o(1).
\]
For the two parts involved in the step, $s^2\le2\sigma_2$, and hence
\[
ms^2\le 2m\sigma_2
\le \frac{2m}{r}\left(1+\frac{29m}{2r}\right)+o(1).
\]
Consequently, the total relative loss in Lemma \ref{lem:balancing} is at most
\[
72\frac{m}{r}+11\frac{m}{r}\left(1+\frac{29m}{2r}\right)+o(1)
=83\frac{m}{r}+\frac{319m^2}{2r^2}+o(1)
\le \frac{1679m}{20r}+o(1),
\]
where the last inequality follows directly from $r\ge168m$. Substituting these estimates into Lemma \ref{lem:balancing} and summing over all balancing steps gives
\begin{equation}
\label{eq:balancing-total}
\inj(H,T_r(n))-\inj(H,K_{\cP})
\ge 2m\left(1-\frac{1679m}{20r}-o(1)\right)Bn^{h-2}.
\end{equation}

Deleting one edge of the host graph destroys at most $2m(n-2)_{h-2}\le 2mn^{h-2}$ injective homomorphisms. Since $|E(G)\setminus E(G_0)|=L$,
\begin{equation}
\label{eq:deletion-loss}
\inj(H,G_0)\ge \inj(H,G)-2mLn^{h-2}.
\end{equation}
Combining \eqref{eq:completion-total} and \eqref{eq:balancing-total}, and using $83/10<1679/20$, gives
\[
\inj(H,T_r(n))-\inj(H,G_0)
\ge 2m\left(1-\frac{1679m}{20r}-o(1)\right)(B+M)n^{h-2}.
\]
Together with \eqref{eq:deletion-loss} and \eqref{eq:Furedi-assembly}, this yields
\begin{align*}
\inj(H,T_r(n))-\inj(H,G)
&\ge 2m\left(\left(1-\frac{1679m}{20r}-o(1)\right)(B+M)-L\right)n^{h-2}\\
&\ge 2m\left(\frac12-\frac{1679m}{20r}-o(1)\right)(B+M)n^{h-2},
\end{align*}
where the second inequality uses $2L\le B+M$. Since $r\ge168m$, the constant part of the coefficient is at least $1/3360$. Hence, for all sufficiently large $n$, the preceding inequality and the extremality of $G$ imply $B+M=0$. Since $B,M\ge0$, we have $B=M=0$, and \eqref{eq:Furedi-assembly} gives $L=0$. Consequently, $G\cong T_r(n)$.\hfill\qedsymbol

\subsection{The stability conclusion}

It suffices to prove the following sequential formulation: if $(G_n)$ is a sequence of $n$-vertex $K_{r+1}$-free graphs satisfying
\begin{equation}
\label{eq:near-extremal-sequence}
\inj(H,G_n)\ge \inj(H,T_r(n))-o(n^h),
\end{equation}
then the edit distance from $G_n$ to $T_r(n)$ is $o(n^2)$.

The equivalence with the definition of Tur\'an-stability will be verified at the end. The next lemma combines the symmetrization argument of Gerbner and Hama Karim \cite{GerbnerHamaKarim} with the degree estimate from the proof of Lemma \ref{lem:calibration}.

\begin{lemma}
\label{lem:cleaning}
There exists a sequence $(\widehat G_n)$ of $n$-vertex $K_{r+1}$-free graphs such that
\[
\inj(H,\widehat G_n)\ge \inj(H,G_n),\qquad
\delta(\widehat G_n)\ge \left(1-\frac{31m}{10r\Delta}\right)n,
\]
and the edit distance between $G_n$ and $\widehat G_n$ is $o(n^2)$.
\end{lemma}

\begin{proof}
Fix $n$ and begin with $Q:=G_n$. Call $v\in V(Q)$ bad if
\[
\Delta\left(1-\frac{d_Q(v)}n\right)\ge \frac{31m}{10r}.
\]
Choose $v_0$ maximizing $I_Q(v_0)$. Since every intermediate graph will have at least as many copies of $H$ as $G_n$, double counting and \eqref{eq:near-extremal-sequence} give
\[
I_Q(v_0)\ge \frac{h}{n}\inj(H,Q)
\ge hp_r(H)n^{h-1}-o(n^{h-1}).
\]

The upper-bound part of the proof of Lemma \ref{lem:calibration} did not use extremality of the host. It used only the induction hypothesis for $H-u$. Therefore, if $v$ is bad, the calculation leading to \eqref{eq:sum-deletion-probabilities} and \eqref{eq:degree-loss-calibrated} gives
\[
I_Q(v)\le \left(hp_r(H)-\frac{m}{23r}\right)n^{h-1}+O_{H,r}(n^{h-2}).
\]
Replace $v$ by a twin of $v_0$. The resulting graph $Q'$ is $K_{r+1}$-free and
\[
\inj(H,Q')-\inj(H,Q)\ge I_Q(v_0)-I_Q(v)-O_H(n^{h-2})\ge \frac{m}{25r}n^{h-1}
\]
for sufficiently large $n$.

Repeat until no bad vertex remains. The exact result already proved implies that every intermediate graph has at most $\inj(H,T_r(n))$ copies. If $s_n$ symmetrization steps are performed, then
\[
\frac{m}{25r}s_n n^{h-1}
\le \inj(H,T_r(n))-\inj(H,G_n)=o(n^h),
\]
and hence $s_n=o(n)$. Each step changes at most $2n$ adjacencies, so the total edit distance is $o(n^2)$. The final graph $\widehat G_n$ has the claimed minimum degree.
\end{proof}

For each $n$, set $\widehat G:=\widehat G_n$. Since $\widehat G$ is $K_{r+1}$-free and $\inj(H,\widehat G)\ge\inj(H,G_n)$, the strict Tur\'an-good conclusion gives
\begin{equation}
\label{eq:stability-deficit}
0\le \inj(H,T_r(n))-\inj(H,\widehat G)
\le \inj(H,T_r(n))-\inj(H,G_n)=o(n^h).
\end{equation}
Use the random variable $Y$ from the proof of Lemma \ref{lem:nonedge-density}, with $\widehat G$ as the host graph. As before, $\E Y=mq(\widehat G)$, and Lemma \ref{lem:cleaning} together with the same second-moment calculation gives
\[
\E Y^2\le mq(\widehat G)
\left(1+mq(\widehat G)+\frac{31m}{5r}+o(1)\right).
\]
On the other hand, \eqref{eq:stability-deficit} and the bound $\inj(H,T_r(n))\ge(1-m/r-o(1))(n)_h$ imply
\[
\Prb(Y>0)=1-\frac{\inj(H,\widehat G)}{(n)_h}
\le \frac{m}{r}+o(1).
\]
Combining these estimates with $\Prb(Y>0)\ge(\E Y)^2/\E Y^2$ gives
\[
q(\widehat G)\le \frac1r\left(1+\frac{29m}{4r}\right)+o(1).
\]

Choose a spanning $r$-partite subgraph $\widehat G_0\subseteq\widehat G$ with the maximum number of edges. Let $\widehat\cP=(A_1,\ldots,A_r)$ be its vertex partition, and let $K_{\widehat\cP}:=K_{A_1,\ldots,A_r}$. Define
\[
\widehat L:=|E(\widehat G)\setminus E(\widehat G_0)|,\qquad
\widehat M:=|E(K_{\widehat\cP})\setminus E(\widehat G_0)|,\qquad
\widehat B:=e(T_r(n))-e(K_{\widehat\cP}).
\]
By Theorem \ref{thm:Furedi}, $2\widehat L\le \widehat B+\widehat M$. The completion and balancing estimates proved above, together with this inequality, give
\[
\inj(H,T_r(n))-\inj(H,\widehat G)
\ge 2m\left(\frac12-\frac{1679m}{20r}-o(1)\right)(\widehat B+\widehat M)n^{h-2}.
\]
The coefficient is positive, whereas the left side is $o(n^h)$. Hence
\[
\widehat B+\widehat M=o(n^2).
\]
It follows that $\widehat L=o(n^2)$. Deleting the $\widehat L$ internal edges and adding the $\widehat M$ missing crossing edges transforms $\widehat G$ into the complete $r$-partite graph $K_{\widehat\cP}$ using $o(n^2)$ edits.

If the part sizes of $K_{\widehat\cP}$ are $a_1,\ldots,a_r$, then
\[
\sum_{i=1}^r\left(a_i-\frac nr\right)^2=2\widehat B+O_r(1)=o(n^2).
\]
Since $r$ is fixed, each $a_i=n/r+o(n)$. Moving $o(n)$ vertices between the parts makes the partition balanced and changes $o(n^2)$ adjacencies. Thus
\[
\dist(\widehat G,T_r(n))=o(n^2).
\]
Lemma \ref{lem:cleaning} now gives
\[
\dist(G_n,T_r(n))=o(n^2).
\]

Finally, if the $\varepsilon$--$\delta$ definition of stability failed, then there would be a fixed $\varepsilon>0$ and a sequence $(G_n)$ satisfying \eqref{eq:near-extremal-sequence} whose edit distance from $T_r(n)$ is at least $\varepsilon n^2$, contradicting the sequential conclusion. This proves the stability assertion and completes the proof of Theorem \ref{thm:main}.\qed

\section{Proof of Theorem~\ref{thm:multipartite-counterexample}}
\label{sec:multipartite-proof}

\begin{proof}[Proof of Theorem \ref{thm:multipartite-counterexample}]
Label the vertices in each part of $T_{r-1}((r-1)t)$ by $1,\dots,t$.
The vertices with the same label form $t$
vertex-disjoint copies of $K_{r-1}$. Two vertices in different parts lie
in a common copy of $K_{r-1}$, while two vertices in the same part lie in
two copies sharing $r-2$ vertices. Hence Theorem \ref{thm:GPS} shows
that $T_{r-1}((r-1)t)$ is strictly $K_r$-Tur\'an-good.

Since $T_{r-1}((r-1)t)$ has $h=t(r-1)$ vertices, to prove that it is not
$K_{r+1}$-Tur\'an-good for large $t$, consider integers $n$ divisible by
$r(r-1)$. Both $T_{r-1}(n)$ and $T_r(n)$ are $K_{r+1}$-free and
have equal-sized parts. An injective homomorphism from
$T_{r-1}((r-1)t)$ to $T_{r-1}(n)$ determines a bijection between their
sets of parts.
Consequently,
\begin{equation}
\label{eq:counter-T-rminus1}
  \inj\bigl(T_{r-1}((r-1)t),T_{r-1}(n)\bigr)
  =(r-1)!\left(\left(\frac{n}{r-1}\right)_t\right)^{r-1}
  =(r-1)!(r-1)^{-h}n^h+O(n^{h-1}).
\end{equation}
A proper coloring of $T_{r-1}((r-1)t)$ with $r$ labeled colors assigns
disjoint nonempty sets of colors to its $r-1$ parts. Thus either every
part is monochromatic and one color is unused, or one part uses two
colors and all other parts are monochromatic. In the first case, there
are $r$ choices for the unused color and $(r-1)!$ bijections from the
remaining colors to the parts, giving $r!$ colorings. In the second
case, there are $r-1$ choices for the part using two colors,
$\binom r2$ choices for the pair of colors, and $(r-2)!$ bijections from
the remaining colors to the remaining parts. Finally, the $t$ vertices
in the two-colored part admit $2^t-2$ assignments using both colors,
since the two monochromatic assignments are excluded. Hence the total
number of proper $r$-colorings of $T_{r-1}((r-1)t)$ is
\begin{equation*}
  (r-1)!\left[r+\binom r2(2^t-2)\right].
\end{equation*}
If the color classes have sizes $N_1,\dots,N_r$, then the number of injective
homomorphisms respecting this coloring is
\[
  \prod_{i=1}^r\left(\frac nr\right)_{N_i}
  =r^{-h}n^h+O(n^{h-1}).
\]
Summing over all proper $r$-colorings gives
\begin{equation}
\label{eq:counter-T-r}
  \inj\bigl(T_{r-1}((r-1)t),T_r(n)\bigr)
  =(r-1)!r^{-h}
  \left[r+\binom r2(2^t-2)\right]n^h+O(n^{h-1}).
\end{equation}
Comparing \eqref{eq:counter-T-rminus1} and \eqref{eq:counter-T-r},
we see that $T_{r-1}(n)$ contains more copies of
$T_{r-1}((r-1)t)$ than $T_r(n)$ for all sufficiently large admissible
$n$ whenever
\begin{equation}
\label{eq:counter-inequality}
  \left(\frac r{r-1}\right)^{t(r-1)}
  >r+\binom r2(2^t-2).
\end{equation}
For $r\ge3$, the binomial theorem gives
\[
  \left(\frac r{r-1}\right)^{r-1}
  =\left(1+\frac1{r-1}\right)^{r-1}
  \ge 2+\frac{r-2}{2(r-1)}\ge\frac94.
\]
The left-hand side of \eqref{eq:counter-inequality} is therefore at
least $(9/4)^t$, whereas its right-hand side is at most $r^2 2^t$.
Thus \eqref{eq:counter-inequality} holds whenever
$(9/8)^t>r^2$, which proves Theorem \ref{thm:multipartite-counterexample}.
\end{proof}

For example, $t=\lceil3\log r/\log(9/8)\rceil$ satisfies
$(9/8)^t>r^2$. The resulting graphs have $O(r\log r)$ vertices and give
$\lambda_{\max}(h)=\Omega(h/\log h)$.

\section{Proof of Theorem \ref{thm:core-satellite}}
\label{sec:core-satellite-proof}

Since $v(F_{r,t})=r-1+t<2(r-1)$, we have
$\lfloor v(F_{r,t})/(r-1)\rfloor=1$. The set $C$ induces $K_{r-1}$,
as does $C_i:=(C\setminus\{y_i\})\cup\{z_i\}$ for every
$1\le i\le t$. Moreover, $|C\cap C_i|=r-2$, and every vertex of
$F_{r,t}$ belongs to $C$ or to some $C_i$. Hence Theorem \ref{thm:GPS}
shows that $F_{r,t}$ is strictly $K_r$-Tur\'an-good.

For a graph $H$ and an integer $q\ge\chi(H)$, let
\[
\Delta_q:=\{(x_1,\ldots,x_q)\in[0,1]^q:x_1+\cdots+x_q=1\},
\]
and let $\Col_q(H)$ be the set of proper colorings of $H$ with labeled colors $1,\ldots,q$. Define the \emph{proper-coloring polynomial}
\[
\Phi_{H,q}(x_1,\ldots,x_q)
:=\sum_{\varphi\in\Col_q(H)}\prod_{v\in V(H)}x_{\varphi(v)}.
\]
Write $\boldsymbol{x}:=(x_1,\ldots,x_q)$. Equivalently, if each vertex independently receives color $i$ with probability $x_i$, then $\Phi_{H,q}(\boldsymbol{x})$ is the probability that the resulting coloring is proper.

Suppose that $a_1+\cdots+a_q=n$ and $a_i=x_i n+O(1)$. For a fixed proper coloring $\varphi$, write $N_i:=|\varphi^{-1}(i)|$. Then
\[
\prod_{i=1}^q(a_i)_{N_i}
=n^{v(H)}\prod_{i=1}^q x_i^{N_i}+O_{H,q}(n^{v(H)-1}).
\]
Summing over all proper colorings gives
\begin{equation}
\label{eq:color-polynomial-asymptotic}
\inj(H,K_{a_1,\ldots,a_q})
=n^{v(H)}\Phi_{H,q}(\boldsymbol{x})+O_{H,q}(n^{v(H)-1}).
\end{equation}
Consequently, if $\Phi_{H,q}(\boldsymbol{x})>\Phi_{H,q}(1/q,\ldots,1/q)$ for some $\boldsymbol{x}\in\Delta_q$, then $H$ is not $K_{q+1}$-Tur\'an-good.

Choose a proper $q$-coloring $\varphi$ uniformly at random from $\Col_q(H)$, and let $N_i:=|\varphi^{-1}(i)|$ and $\mathcal Q:=\E\sum_{i=1}^qN_i^2$.

\begin{lemma}
\label{lem:Hessian-coloring}
Let $q\ge 2$, and let $H$ be a $q$-colorable graph on $h$ vertices. If $q\mathcal Q>h(h+q-1)$, then $(1/q,\ldots,1/q)$ is not a local maximum of $\Phi_{H,q}$.
\end{lemma}

\begin{proof}
Let $u:=1/q$. For $|\tau|<u$, define
\[
\boldsymbol{x}(\tau):=(u+\tau,u-\tau,u,\ldots,u),
\qquad
F(\tau):=\Phi_{H,q}(\boldsymbol{x}(\tau)).
\]
For $\varphi\in\Col_q(H)$, define
\[
f_\varphi(\tau):=(u+\tau)^{N_1}(u-\tau)^{N_2}u^{h-N_1-N_2}.
\]
Thus $F(\tau)=\sum_{\varphi\in\Col_q(H)}f_\varphi(\tau)$. At $\tau=0$,
\[
f'_\varphi(0)=u^{h-1}(N_1-N_2),\qquad
f''_\varphi(0)=u^{h-2}\left((N_1-N_2)^2-N_1-N_2\right).
\]
Interchanging colors $1$ and $2$ gives $F'(0)=0$. By symmetry,
\[
\E N_1=\frac hq,\qquad \E N_1^2=\frac{\mathcal Q}q.
\]
Moreover,
\[
h^2=\E\left(\sum_{i=1}^qN_i\right)^2
=\mathcal Q+q(q-1)\E(N_1N_2),
\]
so
\[
\E(N_1N_2)=\frac{h^2-\mathcal Q}{q(q-1)}.
\]
It follows that
\[
\E\left((N_1-N_2)^2-N_1-N_2\right)
=\frac{2}{q(q-1)}\left(q\mathcal Q-h(h+q-1)\right).
\]
The hypothesis gives $F''(0)>0$, and hence the uniform point is not a local maximum.
\end{proof}

\begin{proof}[Proof of Theorem \ref{thm:core-satellite}]
Choose a proper $r$-coloring of $F_{r,t}$ uniformly at random. The clique $C$ receives $r-1$ distinct colors. Since $\sum_jN_j^2$ is invariant under permuting color labels, label the colors $c_1,\ldots,c_r$ so that $y_i$ receives $c_i$ for $1\le i\le r-1$ and $c_r$ is unused on $C$. Each $z_i$ may receive only $c_i$ or $c_r$. All $2^t$ assignments are proper and equally likely.

For $1\le i\le t$, let $X_i=1$ if $z_i$ receives $c_i$, and let $X_i=0$ otherwise. Define $S:=\sum_{i=1}^tX_i$. Then
\[
N_i=1+X_i\quad(1\le i\le t),
\qquad
N_i=1\quad(t<i\le r-1),
\qquad
N_r=t-S.
\]
The variables $X_i$ are independent, with $\Prb(X_i=0)=\Prb(X_i=1)=1/2$. Hence
\[
\E S=\frac t2,\qquad
\E S^2
=\sum_{i=1}^t\E X_i^2
+2\sum_{1\le i<j\le t}\E(X_iX_j)
=\frac{t(t+1)}4.
\]
Since $X_i^2=X_i$, we also have $\E(1+X_i)^2=5/2$ and $\E(t-S)^2=t(t+1)/4$. Therefore
\[
\mathcal Q=\E\sum_{j=1}^rN_j^2
=\frac52t+(r-1-t)+\frac{t(t+1)}4
=r-1+\frac74t+\frac14t^2.
\]
With $h=v(F_{r,t})=r-1+t$, the condition in Lemma \ref{lem:Hessian-coloring} for $q=r$ is equivalent to $P_r(t)>0$, where
\[
P_r(t):=4\left(r\mathcal Q-h(h+r-1)\right)
=(r-4)t^2-(5r-12)t-4r^2+12r-8.
\]

For $r\ge 17$, let $s:=\sqrt{r-1}\ge 4$ and define $t_0(r):=\lceil 2s+4\rceil$. Then $t_0(r)\le 2s+5<s^2=r-1$.
For $x\ge 2s+4$,
\[
P_r'(x)-P_r'(2s+4)=2(r-4)(x-2s-4)\ge 0.
\]
A direct calculation gives
\[
P_r'(2s+4)=4s^3+3s^2-12s-17>0
\]
for $s\ge 4$, and
\[
P_r(2s+4)=6s^3-12s^2-34s-20>0
\]
for $s\ge 4$. Hence $P_r$ is increasing and positive on $\left[2s+4,\infty\right)$. Thus $P_r(t)>0$ whenever $t_0(r)\le t<r-1$.
Lemma \ref{lem:Hessian-coloring} and \eqref{eq:color-polynomial-asymptotic} show that $F_{r,t}$ is not $K_{r+1}$-Tur\'an-good. This proves the theorem.
\end{proof}

For the lower bound in Theorem \ref{thm:threshold-range}, let $h$ be
sufficiently large, let $r$ be the largest integer satisfying
$r-1+t_0(r)\le h$, and set $t:=t_0(r)$. The graph $F_{r,t}$ has at
most $h$ vertices and is strictly $K_r$-Tur\'an-good but not
$K_{r+1}$-Tur\'an-good. Hence $\lambda_{\max}(h)\ge r+1$. The
maximality of $r$ gives
\[
r-1+t_0(r)\le h<r+t_0(r+1).
\]
Since $t_0(r)=2\sqrt r+O(1)$, it follows that
$h=r+2\sqrt r+O(1)$ and therefore $r=h-2\sqrt h+O(1)$. Consequently,
$\lambda_{\max}(h)\ge h-2\sqrt h-O(1)$, proving the lower bound in
Theorem \ref{thm:threshold-range}. The upper bound follows from
Theorem \ref{thm:main} and $e(H)\le \binom h2$.

\section{Concluding remarks}

Theorem \ref{thm:main} gives the sufficient condition $r\ge168e(H)$, and hence a linear eventual Tur\'an-good and stability threshold for every graph family with $e(H)=O(v(H))$. It is important, however, that its proof establishes an unconditional statement: it proves directly that $H$ is $K_{r+1}$-Tur\'an-good for every $r\ge168e(H)$ and never uses the hypothesis that $H$ is already $K_r$-Tur\'an-good. 
By contrast, the definition of $\lambda(H)$ requires only that, for every $r\ge \lambda(H)$, if $H$ is $K_r$-Tur\'an-good, then it is $K_{r+1}$-Tur\'an-good.
Thus the upper bound supplied by Theorem \ref{thm:main} comes from solving a stronger problem, and the true monotonicity threshold may be substantially smaller. Since $e(H)$ can be quadratic in $v(H)$, it remains open whether the universal monotonicity threshold is $(1+o(1))h$.

\begin{problem}
\label{prob:linear-lambda}
Is $\lambda_{\max}(h)=(1+o(1))h$?
\end{problem}

The lower bound in Theorem \ref{thm:threshold-range} shows that a linear result, if true, would be asymptotically best possible in its order and in the leading constant suggested by the following stronger conjecture of Morrison, Nir, Norin, Rz\k{a}\.zewski and Wesolek \cite{MorrisonEtAl}.

\begin{problem}
\label{prob:vplusone}
Is every graph $H$ $K_{r+1}$-Tur\'an-good whenever $r\ge v(H)+1$?
\end{problem}

An affirmative answer would imply $\lambda_{\max}(h)\le h+1$, while Theorem \ref{thm:threshold-range} shows that the leading constant $1$ would be best possible.

The stability conclusion of Theorem \ref{thm:main} uses the same condition $r\ge 168e(H)$ as the strict extremal conclusion. It is natural to ask whether an analogous result holds for $F$-free graphs with $\chi(F)=r+1$. The present proof uses Theorem \ref{thm:Furedi} to obtain an $r$-partite subgraph, so a different structural input would be needed for a general forbidden graph. Related results were obtained by Gerbner \cite{GerbnerStability,GerbnerWeak} and by Ma and Qiu \cite{MaQiu}.

Following Gerbner \cite{GerbnerWeak}, we say that $H$ is \emph{weakly $K_{r+1}$-Tur\'an-good} if, for every sufficiently large $n$, $\ex(n,H,K_{r+1})=N(H,T)$ for some complete $r$-partite $n$-vertex graph $T$. The construction in Theorem \ref{thm:core-satellite} does not disprove monotonicity for this weaker notion, since its competing graphs are themselves complete $r$-partite. Nevertheless, Theorem \ref{thm:main} implies that the corresponding monotonicity threshold is at most $168e(H)$.

In subsequent work, the present authors prove that there exists an absolute
constant \(C>0\) such that every graph \(H\) with at least one edge is
\(K_{r+1}\)-Tur\'an-good whenever \(r\ge C v(H)\). Thus the general sufficient
threshold admits an upper bound that is linear in the number of vertices of
\(H\). This result and Theorem~\ref{thm:main} are complementary. The
vertex-linear bound is stronger for dense graphs, whereas the explicit bound
\(r\ge 168e(H)\) obtained here is sharper when \(H\) is sparse.
\section*{Acknowledgments}

The authors thank D\'aniel Gerbner for helpful discussions and valuable suggestions. The authors also acknowledge the use of OpenAI's ChatGPT during the preparation of the manuscript. It was used to improve the language, organization, and presentation of the manuscript and to assist in revising preliminary drafts. In particular, the construction in Theorem \ref{thm:core-satellite} arose from an interaction with ChatGPT. The authors take full responsibility for all statements, proofs, constructions, citations, and conclusions presented in this manuscript.

\end{document}